\RequirePackage{fix-cm}
\documentclass[smallextended,envcountreset]{svjour3} 

\smartqed
\usepackage[T1]{fontenc}
\usepackage[utf8]{inputenc}
\usepackage{microtype}
\usepackage{mathtools,amssymb,amsfonts}
\numberwithin{equation}{section}
\usepackage{booktabs,longtable,array}
\usepackage{enumitem}
\usepackage[hidelinks]{hyperref}
\usepackage[nameinlink,noabbrev]{cleveref}

\crefname{appendix}{appendix}{appendices}
\Crefname{appendix}{Appendix}{Appendices}
\setlist{nosep,leftmargin=2em}
\allowdisplaybreaks[2]

\newcommand{\R}{\mathbb{R}}
\newcommand{\C}{\mathbb{C}}

\newcommand{\CM}{\mathcal{CM}}
\newcommand{\Ecal}{\mathcal{E}}
\newcommand{\dd}{\,\mathrm{d}}
\newcommand{\one}{\mathbf{1}}
\newcommand{\Rea}{\operatorname{Re}}

\newcommand{\esssup}{\operatorname*{ess\,sup}}

\journalname{Fract. Calc. Appl. Anal.}
\title{Complete Monotonicity of the Srivastava--Tomovski Function and Its Laplace--Wright Realization: Necessary and Sufficient Conditions}
\titlerunning{Complete monotonicity of the Srivastava--Tomovski function}

\author{Rey R. Cuenca$^{1,2}$ \and
        Welfredo R. Patungan$^{1}$ \and
        Ruzzel Ragas$^{3}$}
\authorrunning{R. R. Cuenca et al.}

\institute{Rey R. Cuenca$^{1,2,*}$ \at
  $^{1}$School of Statistics, University of the Philippines, 1101 Diliman, Quezon City, Philippines \\
  $^{2}$Department of Mathematics and Statistics, Mindanao State University--Iligan Institute of Technology,
  9200 Iligan City, Philippines \\
  \email{rey.cuenca@g.msuiit.edu.ph} $^*$ corresponding author
  \and
  Welfredo R. Patungan$^{1}$ \at
  School of Statistics, University of the Philippines, 1101 Diliman, Quezon City, Philippines
  \and
  Ruzzel Ragas$^{3}$ \at
  School of Mathematics and Statistics, The University of Sydney, Carslaw F07, 2006 NSW, Australia
}
\date{}

\hypersetup{
  pdftitle={Complete Monotonicity of the Srivastava--Tomovski Function and Its Laplace--Wright Realization: Necessary and Sufficient Conditions},
  pdfauthor={Rey R. Cuenca, Welfredo R. Patungan, Ruzzel Ragas},
  pdfsubject={Necessary and sufficient conditions for complete monotonicity and analytic phases of a four-parameter Mittag--Leffler function},
  pdfkeywords={complete monotonicity, Srivastava--Tomovski function, Prabhakar function, Wright function, Bernstein measure, generalized Wright series}
}

\begin{document}
\maketitle

\begin{abstract}
We determine necessary and sufficient conditions for complete monotonicity of
the Srivastava--Tomovski extension, a generalized Mittag--Leffler kernel
arising in fractional calculus, while separating the defining
generalized-Wright series from its positive-axis realization when the series
is not entire. For positive parameters, let
\(\Delta=1+\alpha-\kappa\). In the entire regime \(\Delta>0\),
\(E_{\alpha,\beta}^{\gamma,\kappa}(-x)\) is completely monotone on
\((0,\infty)\) if and only if
\(\alpha\leq\kappa\) and \(\kappa\beta\geq\alpha\gamma\).
For arbitrary positive parameters, the corresponding positive-axis
realization satisfies the same criterion; when
\(0<\alpha<\kappa\), it is defined by a second-kind Wright kernel and
remains meaningful through the finite-radius and zero-radius phases of the
defining series. We also identify the Bernstein measure in every admissible
case. In the interior region it is, after normalization, the pushforward under
\(t\mapsto t^\kappa\) of a power-biased Wright distribution; at
\(\alpha=\kappa\) it becomes a powered beta law for \(\beta>\gamma\)
and the atom \(\Gamma(\gamma)^{-1}\delta_1\) for \(\beta=\gamma\).
Necessity in the excluded regions follows from a Mellin-transform zero,
uniqueness of finite signed Laplace transforms, and moment-support
asymptotics. The Prabhakar criterion is recovered by setting \(\kappa=1\).
\end{abstract}

\section{Introduction}

A smooth function \(f:(0,\infty)\to\R\) is completely monotone if
\((-1)^n f^{(n)}(x)\geq0\) for every \(n\geq0\) and \(x>0\).
Bernstein's theorem identifies this derivative condition with a positive
Laplace representation,
\[
f(x)=\int_{[0,\infty)}e^{-xu}\,M(\dd u),\qquad M\geq0.
\]
Complete monotonicity is therefore both a sign property and a
positive-measure representation \cite{SchillingSongVondracek2012}.

Classical complete-monotonicity results for Mittag--Leffler functions go back
to Pollard \cite{Pollard1948} and Schneider \cite{Schneider1996}. Prabhakar's
three-parameter extension \cite{Prabhakar1971} became a standard kernel in
generalized fractional calculus; see, for example,
\cite{KilbasSaigoSaxena2004,MainardiGarrappa2015,Garrappa2015,GarraGarrappa2018,GiustiEtAl2020,GorenfloEtAl2020}.
Shukla and Prajapati introduced the generalized-Pochhammer step deformation
for restricted step values \cite{ShuklaPrajapati2007}. Srivastava and Tomovski
subsequently gave a substantially improved and generalized formulation with a
general step parameter \(\kappa\) and incorporated it into a
fractional-calculus integral operator \cite{SrivastavaTomovski2009}. Thus the
four-parameter extension studied here has a direct fractional-calculus origin.

The step parameter creates an analytic issue that is easy to conceal if the
series and a Laplace representation are denoted by the same symbol. The
defining generalized-Wright series can be entire, have finite radius of
convergence, or have radius zero. A statement on the whole positive axis must
therefore distinguish the original series from any separately defined
Laplace realization outside their common analytic domain. There is a second,
independent issue. Positivity of a displayed inverse kernel proves
sufficiency for complete monotonicity, but loss of positivity of that kernel
does not by itself exclude another positive Bernstein measure. Necessity
requires a uniqueness or moment argument. The boundary
\(\alpha=\kappa\) also has to be treated separately because the representing
measure changes from a Wright density to a compact beta-type law and then to
an atom.

Exact complete-monotonicity characterizations have a parallel history for
neighboring Mittag--Leffler families. Schneider proved that the
two-parameter function \(E_{\alpha,\beta}(-x)\) is completely monotone if and
only if \(0<\alpha\leq1\) and \(\beta\geq\alpha\)
\cite{Schneider1996}. Boudabsa and Simon obtained an if-and-only-if criterion
for the three-parameter Kilbas--Saigo function by probabilistic and Mellin
methods \cite{BoudabsaSimon2021}. For a different Le Roy-type deformation,
Garrappa, Rogosin, and Mainardi raised the parameter-range problem in FCAA
\cite{GarrappaRogosinMainardi2017}; G{\'o}rska, Horzela, and Garrappa then
derived several complete-monotonicity regions by Laplace methods in FCAA
\cite{GorskaHorzelaGarrappa2019}. Simon subsequently obtained exact criteria
on substantial parameter strata, together with complementary necessary and
sufficient conditions in the remaining regime \cite{Simon2022}. The
Le Roy-type family is not, in general, a specialization of the
Srivastava--Tomovski family. Its relevance here is methodological: these
results show why a sharp characterization needs necessity arguments in
addition to a positive integral representation.

For the Prabhakar function, complete-monotonicity and Laplace-representation
results are available in several forms
\cite{TomovskiPoganySrivastava2014,MainardiGarrappa2015,GorskaEtAl2021}.
In particular, G{\'o}rska, Horzela, Lattanzi, and Pog{\'a}ny prove
sufficiency under \(0<\alpha\leq1\) and
\(\beta\geq\alpha\gamma\)
\cite[Theorem~1, p.~120]{GorskaEtAl2021}. Ferreira and Simon identify the
probabilistic second-kind Wright densities, their Mellin moments, and the
obstructions outside the probabilistic parameter range
\cite[Eqs.~(3)--(4) and the discussion on p.~310]{FerreiraSimon2024}.
The equal-step beta and atomic limits are consistent with the beta-power and
zero-balanced Fox--\(H\) constructions in
\cite{KarpPrilepkina2016,KarpPrilepkina2017,BeghinCristofaroDaSilva2025}. Related positivity results for generalized
Wright and Fox--\(H\) functions appear in \cite{Mehrez2021,Giraldi2026}.

We distinguish the defining series from its positive-axis realization and
obtain the following necessary and sufficient conditions. In the entire regime the
Srivastava--Tomovski series is completely monotone precisely when
\[
\alpha\leq\kappa,\qquad \kappa\beta\geq\alpha\gamma,
\]
and the same inequalities characterize the independently defined
Laplace--Wright realization in all analytic phases. For
\(0<\alpha<\kappa\), the Bernstein measure is a powered pushforward of a
power-biased Wright law. The excluded interior region is ruled out by a zero
of the Mellin transform together with uniqueness of finite signed Laplace
transforms; the equal-step and supercritical exclusions follow from
moment-support asymptotics. The specialization \(\kappa=1\) yields the
Prabhakar criterion. This statement concerns the unweighted functions and
realizations considered here; weighted relaxation kernels such as
\(t^{\beta-1}E_{\alpha,\beta}^{\gamma}(-t^\alpha)\) are different objects and
can have different admissible parameter regions.

	\section{The series and the positive-axis realization}
	
	\begin{definition}[Srivastava--Tomovski series]\label{def:series}
		For \(\alpha,\beta,\gamma,\kappa>0\), set
		\begin{equation}\label{eq:series}
			E_{\alpha,\beta}^{\gamma,\kappa}(z)
			:=\sum_{n=0}^{\infty}
			\frac{\Gamma(\gamma+\kappa n)}
			{\Gamma(\gamma)\,n!\,\Gamma(\beta+\alpha n)}z^n
		\end{equation}
		first as a formal power series.  When its radius of convergence is positive,
		it defines an analytic function on the open disk of convergence.  At a
		boundary point where the series converges, its sum is understood only as a
		pointwise boundary value unless an analytic continuation is separately
		established.  Write
		\begin{equation}\label{eq:coordinates}
			\Delta:=1+\alpha-\kappa,
			\qquad
			a:=\frac{\alpha}{\kappa},
			\qquad
			\mu:=\beta-\frac{\alpha\gamma}{\kappa}.
		\end{equation}
	\end{definition}
	
	Thus \eqref{eq:series} is an entire function when \(\Delta>0\), an analytic
	germ on \(\lvert z\rvert<R\) when \(\Delta=0\), and only a formal series when
	\(\Delta<0\).  In the last case it does not define an analytic germ at zero.
	
	This is the generalized-Pochhammer normalization of
	\cite[Eq.~(1.13), p.~200]{SrivastavaTomovski2009}; in particular,
	\begin{equation}\label{eq:value-zero}
		E_{\alpha,\beta}^{\gamma,\kappa}(0)=\frac{1}{\Gamma(\beta)}.
	\end{equation}
	The generalized-Wright convergence criterion gives the following trichotomy.
	
	\begin{proposition}[Convergence trichotomy]\label{prop:convergence}
		The series \eqref{eq:series} is entire if \(\Delta>0\), has radius
		\begin{equation}\label{eq:radius}
			R=\frac{\alpha^\alpha}{\kappa^\kappa}
		\end{equation}
		if \(\Delta=0\), and has radius zero if \(\Delta<0\).
	\end{proposition}
	
	The classical generalized-Wright criterion may be found in
	\cite[Theorem~1, pp.~201--202]{SrivastavaTomovski2009} and
	\cite[Section~2.3]{ParisKaminski2001}; a direct coefficient proof is
	included in \Cref{app:phases}.
	
	For \(\lambda>-1\) and \(\nu,z\in\C\), the Wright function is
	\begin{equation}\label{eq:wright}
		W_{\lambda,\nu}(z)
		:=\sum_{n=0}^{\infty}\frac{z^n}{n!\Gamma(\lambda n+\nu)}.
	\end{equation}
	When \(0<a<1\), the function \(W_{-a,\nu}\) is of the second kind
	\cite{Mainardi2010,ParisConsiglioMainardi2021}.
	
	\begin{definition}[Positive-axis realization]\label{def:realization}
		For \(0<\alpha<\kappa\), define
		\begin{align}
			K_{\alpha,\beta}^{\gamma,\kappa}(u)
			&:=\frac{u^{\gamma/\kappa-1}}{\kappa\Gamma(\gamma)}
			W_{-\alpha/\kappa,\,\mu}(-u^{1/\kappa}),\label{eq:kernel}\\
			\Ecal_{\alpha,\beta}^{\gamma,\kappa}(x)
			&:=\int_0^\infty e^{-xu}
			K_{\alpha,\beta}^{\gamma,\kappa}(u)\dd u,
			\qquad x\geq0.\label{eq:realization}
		\end{align}
		For \(\alpha\geq\kappa\), define instead
		\begin{equation}\label{eq:realization-series}
			\Ecal_{\alpha,\beta}^{\gamma,\kappa}(x)
			:=E_{\alpha,\beta}^{\gamma,\kappa}(-x).
		\end{equation}
	\end{definition}
	
	The series in \eqref{eq:realization-series} is entire because
	\(\alpha\geq\kappa\) implies \(\Delta\geq1\).  For
	\(0<\alpha<\kappa\), the kernel in \eqref{eq:kernel} is integrable in
	absolute value for every real \(\mu\); it need not be nonnegative.
	
	\begin{proposition}[Agreement and analytic phases]\label{prop:phases}
		Suppose \(0<\alpha<\kappa\).
		\begin{enumerate}[label=\textup{(\roman*)}]
			\item If \(\Delta>0\), the integral \eqref{eq:realization} extends to an
			entire function of \(x\) and equals
			\(E_{\alpha,\beta}^{\gamma,\kappa}(-x)\) for every \(x\in\C\).
			\item If \(\Delta=0\), the integral is holomorphic on
			\(\Rea x>-R\), agrees with
			\(E_{\alpha,\beta}^{\gamma,\kappa}(-x)\) for \(\lvert x\rvert<R\),
			and therefore analytically continues that germ to the half-plane
			\(\Rea x>-R\).
			\item If \(\Delta<0\), the integral is holomorphic on \(\Rea x>0\),
			has finite right derivatives of every order at \(x=0\), and is not
			analytic there. It is therefore a separately defined positive-axis
			realization rather than an analytic continuation of the zero-radius
			formal series.
		\end{enumerate}
	\end{proposition}
	
	The proof, including the exact critical decay rate, is given in
	\Cref{app:phases}.  This proposition is the reason for retaining distinct
	symbols \(E\) and \(\Ecal\).
	
	\section{Complete-monotonicity theorems}
	
	Let \(\CM(0,\infty)\) denote the cone of completely monotone functions on
	\((0,\infty)\).  This is a property on the open half-line; continuity at zero
	is not part of its definition.
	
	\begin{theorem}[Entire-series characterization]\label{thm:series}
		Let \(\alpha,\beta,\gamma,\kappa>0\) and assume
		\(\Delta=1+\alpha-\kappa>0\).  Then
		\begin{equation}\label{eq:series-iff}
			E_{\alpha,\beta}^{\gamma,\kappa}(-\,\cdot)
			\in\CM(0,\infty)
			\quad\Longleftrightarrow\quad
			\alpha\leq\kappa
			\quad\text{and}\quad
			\kappa\beta\geq\alpha\gamma.
		\end{equation}
	\end{theorem}
	
	\begin{theorem}[Positive-axis realization characterization]\label{thm:realization}
		For arbitrary \(\alpha,\beta,\gamma,\kappa>0\), the realization in
		\Cref{def:realization} satisfies
		\begin{equation}\label{eq:realization-iff}
			\Ecal_{\alpha,\beta}^{\gamma,\kappa}
			\in\CM(0,\infty)
			\quad\Longleftrightarrow\quad
			\alpha\leq\kappa
			\quad\text{and}\quad
			\kappa\beta\geq\alpha\gamma.
		\end{equation}
		Every admissible case is strictly completely monotone.  Its Bernstein
		measure is as follows.
		
		\begin{enumerate}[label=\textup{(\roman*)}]
			\item If \(0<\alpha<\kappa\) and \(\mu\geq0\), the measure is
			absolutely continuous with the strictly positive density
			\(K_{\alpha,\beta}^{\gamma,\kappa}\) in \eqref{eq:kernel}.
			\item If \(\alpha=\kappa\) and \(\beta>\gamma\), the density is
			\begin{equation}\label{eq:beta-density}
				K_{=}(u)=
				\frac{u^{\gamma/\kappa-1}
					(1-u^{1/\kappa})^{\beta-\gamma-1}}
				{\kappa\Gamma(\gamma)\Gamma(\beta-\gamma)}
				\one_{(0,1)}(u).
			\end{equation}
			\item If \(\alpha=\kappa\) and \(\beta=\gamma\), the measure is
			\begin{equation}\label{eq:atom}
				\frac{1}{\Gamma(\gamma)}\,\delta_1,
				\qquad
				\Ecal_{\kappa,\gamma}^{\gamma,\kappa}(x)
				=\frac{e^{-x}}{\Gamma(\gamma)}.
			\end{equation}
		\end{enumerate}
		In every case the raw measure has total mass \(1/\Gamma(\beta)\).
		Consequently, every admissible realization extends continuously to zero and
		has value \(1/\Gamma(\beta)\) there.
	\end{theorem}
	
	The measures in this theorem are obtained from established probability laws by
	deterministic transformations. In the interior, the normalized measure is the
	pushforward of a power bias of the Ferreira--Simon Wright law; see
	\Cref{rem:probability-law}. The equal-step cases agree with the beta-power and
	zero-balanced atomic laws in
	\cite[Remark~4, p.~150]{KarpPrilepkina2016},
	\cite[Corollary~2.1, p.~356]{KarpPrilepkina2017}, and
	\cite[Lemma~5 and Theorem~8(C0,C2), pp.~10--14]{BeghinCristofaroDaSilva2025};
	see also \Cref{prop:boundary}.
	
	The inequalities in \eqref{eq:realization-iff} are equivalently
	\(0<\alpha<\kappa,\mu\geq0\) or
	\(\alpha=\kappa,\beta\geq\gamma\).  The complete partition is displayed in
	\Cref{tab:strata}.
	
	\begin{longtable}{@{}>{\raggedright\arraybackslash}p{0.26\textwidth}
			>{\raggedright\arraybackslash}p{0.12\textwidth}
			>{\raggedright\arraybackslash}p{0.54\textwidth}@{}}
		\caption{Complete-monotonicity status on the seven positive-parameter strata.}
		\label{tab:strata}\\
		\toprule
		Stratum & Status & Representing measure or obstruction\\
		\midrule
		\endfirsthead
		\toprule
		Stratum & Status & Representing measure or obstruction\\
		\midrule
		\endhead
		\bottomrule
		\endfoot
		\(0<\alpha<\kappa,\ \mu>0\)
		& Strictly CM
		& Strictly positive Wright density \eqref{eq:kernel}\\
		\(0<\alpha<\kappa,\ \mu=0\)
		& Strictly CM
		& Cancelled Wright density; its first Wright coefficient vanishes but the density remains strictly positive\\
		\(0<\alpha<\kappa,\ \mu<0\)
		& Not CM
		& Zero positive Mellin ordinate forces a sign-changing inverse; finite signed Laplace uniqueness excludes another positive measure\\
		\(\alpha=\kappa,\ \beta>\gamma\)
		& Strictly CM
		& Powered-beta density \eqref{eq:beta-density}\\
		\(\alpha=\kappa,\ \beta=\gamma\)
		& Strictly CM
		& Atom \(\Gamma(\gamma)^{-1}\delta_1\)\\
		\(\alpha=\kappa,\ \beta<\gamma\)
		& Not CM
		& Moment roots force support in \([0,1]\), whereas the moments diverge\\
		\(\alpha>\kappa\)
		& Not CM
		& Moment roots tend to zero, forcing support at zero and contradicting the positive first moment\\
	\end{longtable}
	
	\section{The representing measures}
	
	The next results give the representing measures and the transforms used in the necessity arguments.
	
	\begin{lemma}[Second-kind Wright Mellin transform]\label{lem:wright-mellin}
		Let \(0<a<1\) and \(\nu\in\R\). If
		\(\nu\notin\mathbb Z_{\leq0}\), then for \(\Rea z>0\),
		\begin{equation}\label{eq:wright-mellin}
			\int_0^\infty t^{z-1}W_{-a,\nu}(-t)\dd t
			=\frac{\Gamma(z)}{\Gamma(\nu+az)},
		\end{equation}
		and the integral is absolutely convergent. If
		\(\nu\in\mathbb Z_{\leq0}\), the same identity holds for
		\(\Rea z>-1\), with the quotient at \(z=0\) interpreted by continuity.
		These lower bounds are sharp for absolute convergence at the origin.
		Moreover, for fixed \(a,\nu\) there exist \(C,N,b>0\) such that
		\begin{equation}\label{eq:wright-tail}
			\lvert W_{-a,\nu}(-t)\rvert
			\leq Ct^N\exp\{-bt^{1/(1-a)}\},
			\qquad t\geq1.
		\end{equation}
	\end{lemma}
	The gamma quotient in \eqref{eq:wright-mellin} is classical at the level of
	Mellin--Barnes inversion. In particular, Luchko's representation
	\cite[Eq.~(4), p.~245]{Luchko2019} has kernel
	\(\Gamma(s)/\Gamma(\beta-\rho s)\); the substitution
	\(\rho=-a\), \(\beta=\nu\) gives \(\Gamma(s)/\Gamma(\nu+as)\).
	The contour representation identifies the classical gamma quotient but does not
	by itself determine the sharp strip of the ordinary Mellin integral. In the
	probabilistic range \(\nu\geq0\), Ferreira and Simon give the normalized density
	and its ordinary Mellin moments
	\cite[Eqs.~(3)--(4), p.~310]{FerreiraSimon2024}. Appendix~\ref{app:mellin}
	establishes the ordinary integral for every real \(\nu\), including the exact
	lower boundary and the removable value at \(z=0\) when
	\(\nu\in\mathbb Z_{\leq0}\). Thus the classical component is the
	Mellin--Barnes gamma quotient; the extension needed here is the ordinary-integral
	statement on its sharp absolute-convergence domain.
	
	\begin{proposition}[Interior kernel and moments]\label{prop:kernel}
		Suppose \(0<\alpha<\kappa\). Then
		\(K_{\alpha,\beta}^{\gamma,\kappa}\in L^1(0,\infty)\), and
		\(\int_0^\infty u^n\lvert K_{\alpha,\beta}^{\gamma,\kappa}(u)\rvert\dd u<\infty\)
		for every \(n\in\mathbb N_0\). Moreover,
		\begin{equation}\label{eq:kernel-mellin}
			\int_0^\infty u^{s-1}
			K_{\alpha,\beta}^{\gamma,\kappa}(u)\dd u
			=\frac{\Gamma(\gamma+\kappa(s-1))}
			{\Gamma(\gamma)\Gamma(\beta+\alpha(s-1))}
		\end{equation}
		on the maximal fundamental strip
		\begin{equation}\label{eq:kernel-strip}
			\Rea s>
			\begin{cases}
				1-\gamma/\kappa,&\mu\notin\mathbb Z_{\leq0},\\
				1-(\gamma+1)/\kappa,&\mu\in\mathbb Z_{\leq0}.
			\end{cases}
		\end{equation}
		In particular,
		\begin{align}
			\int_0^\infty K_{\alpha,\beta}^{\gamma,\kappa}(u)\dd u
			&=\frac1{\Gamma(\beta)},\label{eq:kernel-mass}\\
			\int_0^\infty u^nK_{\alpha,\beta}^{\gamma,\kappa}(u)\dd u
			&=\frac{\Gamma(\gamma+\kappa n)}
			{\Gamma(\gamma)\Gamma(\beta+\alpha n)}.
			\label{eq:kernel-moments}
		\end{align}
		If \(\mu\geq0\), the kernel is strictly positive on \((0,\infty)\).
	\end{proposition}
	
	\begin{proof}
		Put \(u=t^\kappa\).  Then
		\begin{equation}\label{eq:pushforward}
			K_{\alpha,\beta}^{\gamma,\kappa}(u)\dd u
			=\frac{t^{\gamma-1}}{\Gamma(\gamma)}
			W_{-a,\mu}(-t)\dd t.
		\end{equation}
		The strip \eqref{eq:kernel-strip} follows by putting
		\(z=\gamma+\kappa(s-1)\) in the two cases of
		\Cref{lem:wright-mellin}.
		Local analyticity at zero and \eqref{eq:wright-tail} give absolute
		integrability and all absolute moments.  Applying
		\eqref{eq:wright-mellin} with
		\(z=\gamma+\kappa(s-1)\) proves \eqref{eq:kernel-mellin}; the remaining
		identities follow by taking \(s=1\) and \(s=n+1\).
		
		For \(0<a<1\) and \(\mu\geq0\), Ferreira and Simon identify
		\(\Gamma(a+\mu)W_{-a,\mu}(-t)\) as a probability density and prove the
		corresponding positive-zero classification
		\cite[Eqs.~(3)--(4), p.~310, and Section~6.2, p.~329]{FerreiraSimon2024}. Hence
		the density is strictly positive.  At \(\mu=0\), this conclusion also follows
		directly from
		\begin{equation}\label{eq:cancellation}
			W_{-a,0}(-t)=at\,W_{-a,1-a}(-t),
		\end{equation}
		obtained by shifting the defining series.  The positive factors in
		\eqref{eq:pushforward} preserve the sign.
	\end{proof}
	
	\begin{remark}[Probability-law interpretation]
		\label{rem:probability-law}
		Assume \(\mu\geq0\), and let \(M_{a,\mu}\) denote the Ferreira--Simon
		variable with density
		\[
			p_{a,\mu}(t)=\Gamma(a+\mu)W_{-a,\mu}(-t),\qquad t>0.
		\]
		Since \(\gamma-1>-1\), their moment formula gives
		\begin{equation}\label{eq:power-bias-normalizer}
			\mathbb E[M_{a,\mu}^{\gamma-1}]
			=\frac{\Gamma(\gamma)\Gamma(a+\mu)}
			{\Gamma(a+\mu+a(\gamma-1))}
			=\frac{\Gamma(\gamma)\Gamma(a+\mu)}{\Gamma(\beta)},
		\end{equation}
		because \(\mu=\beta-a\gamma\).  Therefore the
		\((\gamma-1)\)-power bias of \(M_{a,\mu}\) has density
		\begin{equation}\label{eq:power-biased-density}
			p^{[\gamma-1]}_{a,\mu}(t)
			=\frac{\Gamma(\beta)}{\Gamma(\gamma)}
			t^{\gamma-1}W_{-a,\mu}(-t).
		\end{equation}
		Combining \eqref{eq:power-biased-density} with \eqref{eq:pushforward} yields
		\begin{equation}\label{eq:power-bias-pushforward}
			\Gamma(\beta)K_{\alpha,\beta}^{\gamma,\kappa}(u)\dd u
			=(t\mapsto t^\kappa)_{\#}
			\bigl[p^{[\gamma-1]}_{a,\mu}(t)\dd t\bigr].
		\end{equation}
		Thus the probability-law interpretation in the admissible interior follows
		directly from the Ferreira--Simon density and moment formula. For
		\(\mu<0\), the same algebraic kernel is used only as a finite signed density;
		Ferreira and Simon already note the Mellin-zero obstruction to a nonnegative
		Wright density outside their probabilistic range
		\cite[discussion following Eq.~(4), p.~310]{FerreiraSimon2024}.
	\end{remark}
	
	If \(\mu\geq0\), \eqref{eq:realization} is therefore the Laplace transform of
	a finite positive measure.  Bernstein's theorem gives complete monotonicity,
	and differentiation under the integral gives, for every \(n\geq0\),
	\begin{equation}\label{eq:strict-derivative}
		(-1)^n\frac{\dd^n}{\dd x^n}
		\Ecal_{\alpha,\beta}^{\gamma,\kappa}(x)
		=\int_0^\infty u^ne^{-xu}K_{\alpha,\beta}^{\gamma,\kappa}(u)\dd u>0.
	\end{equation}
	
	\begin{proposition}[The boundary \(\alpha=\kappa\)]\label{prop:boundary}
		If \(\alpha=\kappa\) and \(\beta>\gamma\), then
		\begin{equation}\label{eq:beta-laplace}
			E_{\kappa,\beta}^{\gamma,\kappa}(-x)
			=\int_0^1e^{-xu}K_{=}(u)\dd u,
		\end{equation}
		where \(K_=\) is given by \eqref{eq:beta-density}.  After multiplication
		by \(\Gamma(\beta)\), this is the law of \(U=T^\kappa\) with
		\(T\sim\operatorname{Beta}(\gamma,\beta-\gamma)\).  If
		\(\beta=\gamma\), the representation is the atom \eqref{eq:atom}.
	\end{proposition}
	
	\begin{proof}
		For \(\beta>\gamma\), substitute \(u=t^\kappa\) in the moments of
		\eqref{eq:beta-density}.  Euler's beta integral gives
		\[
		\int_0^1u^nK_=(u)\dd u
		=\frac{\Gamma(\gamma+\kappa n)}
		{\Gamma(\gamma)\Gamma(\beta+\kappa n)}.
		\]
		The exponential series is absolutely integrable on the compact support, and
		\eqref{eq:beta-laplace} follows.  When \(\beta=\gamma\), the two gamma
		factors in every coefficient of \eqref{eq:series} cancel, giving
		\(e^{-x}/\Gamma(\gamma)\).  The latter case is atomic and is not
		obtained by substituting \(\beta=\gamma\) into the singular beta density.
		
		These endpoints follow from the Ferreira--Simon \(a=1\) base law by the same
		power-bias and deterministic-pushforward calculation as in
		\Cref{rem:probability-law}. At \(a=1\), their base variable is
		\(\operatorname{Beta}(1,\beta-\gamma)\); its \((\gamma-1)\)-power bias is
		\(\operatorname{Beta}(\gamma,\beta-\gamma)\), and
		\(t\mapsto t^\kappa\) gives \eqref{eq:beta-density}. In the notation of
		Karp and Prilepkina this is the specialization
		\[
		(A_1,\alpha_1,\beta_1)
		=(\kappa,\gamma,\beta-\gamma)
		\]
		of their beta-power law
		\cite[Remark~4, p.~150]{KarpPrilepkina2016}.  In the class (C2) of Beghin,
		Cristofaro, and da Silva the same normalized law is obtained from
		\[
		(p,\eta,e,b_0,K_2)
		=\left(1,\kappa,\gamma-\kappa,\beta-\gamma,
		\frac{\Gamma(\gamma)}{\Gamma(\beta)}\right),
		\]
		while their class (C0) contains the atom
		\cite[Theorem~8(C0,C2), pp.~13--14]{BeghinCristofaroDaSilva2025}.
		To verify the endpoint limit, let \(\varepsilon=\beta-\gamma\downarrow0\) and
		\(T_\varepsilon\sim\operatorname{Beta}(\gamma,\varepsilon)\).  Since
		\[
		\mathbb E(1-T_\varepsilon)
		=\frac{\varepsilon}{\gamma+\varepsilon}\longrightarrow0,
		\]
		we have \(T_\varepsilon^\kappa\to1\) in probability and hence
		\begin{equation}\label{eq:beta-atom-limit}
			\frac{1}{\Gamma(\gamma+\varepsilon)}
			\operatorname{Law}(T_\varepsilon^\kappa)
			\Longrightarrow \frac{1}{\Gamma(\gamma)}\delta_1.
		\end{equation}
		This weak zero-balanced transition is also covered by
		\cite[Corollary~2.1, p.~356, and Theorem~3,
		pp.~357--361]{KarpPrilepkina2017}.
	\end{proof}
	
	The preceding propositions prove every admissible row of
	\Cref{tab:strata}, including strict complete monotonicity.
	
	\section{Necessity}
	
	The three excluded regions as listed in Table~\ref{tab:strata} require different arguments.
	
	\subsection[Negative mu in the interior]{Negative \texorpdfstring{\(\mu\)}{mu} in the interior:
		\texorpdfstring{\(0<\alpha<\kappa\) and \(\mu<0\)}{alpha below kappa and negative mu}}

	The Mellin-zero obstruction for a Wright density with negative second parameter
	is already noted by Ferreira and Simon
	\cite[discussion following Eq.~(4), p.~310]{FerreiraSimon2024}. Here it is
	combined with the pushforward identity and finite signed Laplace uniqueness to
	rule out any alternative positive Bernstein measure for the present realization.
	
	\begin{proposition}\label{prop:negative-mu}
		If \(0<\alpha<\kappa\) and \(\mu<0\), then
		\(\Ecal_{\alpha,\beta}^{\gamma,\kappa}\notin\CM(0,\infty)\).
	\end{proposition}
	
	\begin{proof}
		Set \(z_0=-\mu/a>0\).  The integral in
		\eqref{eq:wright-mellin} is absolutely convergent at \(z_0\), and
		\begin{equation}\label{eq:mellin-zero}
			\int_0^\infty t^{z_0-1}W_{-a,\mu}(-t)\dd t
			=\frac{\Gamma(z_0)}{\Gamma(0)}=0,
		\end{equation}
		where \(1/\Gamma(0)=0\).  The analytic function
		\(W_{-a,\mu}(-t)\) is not identically zero: if every series coefficient
		vanished, then both \(\mu\) and \(\mu-a\) would be nonpositive integers,
		which would force \(a\) to be an integer, contrary to \(0<a<1\).  Hence it cannot be
		nonnegative on \((0,\infty)\), since the weight in
		\eqref{eq:mellin-zero} is strictly positive.  It cannot be nonpositive
		either: by \eqref{eq:pushforward} and \eqref{eq:kernel-mass},
		\[
		\int_0^\infty t^{\gamma-1}W_{-a,\mu}(-t)\dd t
		=\frac{\Gamma(\gamma)}{\Gamma(\beta)}>0.
		\]
		Thus the Wright function, and therefore \(K\), has both signs.  Since it is
		continuous and nonzero, its positive and negative sets each contain a
		nonempty open interval.
		
		Suppose nevertheless that \(\Ecal\) were completely monotone.  Bernstein's
		theorem would give a positive representing measure \(N\).  For the signed
		kernel, dominated convergence (using \(K\in L^1\)) gives
		\(\Ecal(x)\to\int K=1/\Gamma(\beta)\) as \(x\downarrow0\).  For \(N\),
		monotone convergence gives the same limit, so
		\(N([0,\infty))=1/\Gamma(\beta)<\infty\).  The finite signed measure
		\(K(u)\dd u-N(\dd u)\) has zero Laplace transform.  The signed uniqueness
		lemma in \Cref{app:measure} forces \(N(\dd u)=K(u)\dd u\), contradicting the
		negative part of \(K\).
	\end{proof}
	
	\subsection{The boundary below the beta range}
	
	For \(\alpha\geq\kappa\), the defining series is entire.  If its negative-axis
	restriction were completely monotone with Bernstein measure \(N\), then
	differentiating its Laplace representation for \(x>0\) is legitimate because
	\(u^ne^{-xu}\leq C_{n,x}e^{-xu/2}\).  Applying monotone convergence to
	\(u^ne^{-xu}\) as \(x\downarrow0\) then identifies every moment:
	\begin{equation}\label{eq:hypothetical-moments}
		m_n:=\int_{[0,\infty)}u^nN(\dd u)
		=\frac{\Gamma(\gamma+\kappa n)}
		{\Gamma(\gamma)\Gamma(\beta+\alpha n)}.
	\end{equation}
	The support lemma in \Cref{app:measure} says that, for a finite positive measure,
	\(\lim m_n^{1/n}\) equals the upper endpoint of its support when the limit is
	finite.
	
	\begin{proposition}\label{prop:boundary-excluded}
		If \(\alpha=\kappa\) and \(\beta<\gamma\), then
		\(E_{\kappa,\beta}^{\gamma,\kappa}(-\,\cdot)\) is not completely
		monotone.
	\end{proposition}
	
	\begin{proof}
		Stirling's formula gives
		\begin{equation}\label{eq:equal-moment-asymptotic}
			m_n\sim\frac{(\kappa n)^{\gamma-\beta}}{\Gamma(\gamma)},
			\qquad m_n^{1/n}\longrightarrow1.
		\end{equation}
		The hypothetical measure would therefore be supported in \([0,1]\), so
		\(m_n\leq m_0\) for all \(n\).  This contradicts
		\(m_n\to\infty\) in \eqref{eq:equal-moment-asymptotic}.
	\end{proof}
	
	\subsection[The supercritical step]{The supercritical step:
		\texorpdfstring{\(\alpha>\kappa\)}{alpha above kappa}}
	
	\begin{proposition}\label{prop:supercritical}
		If \(\alpha>\kappa\), then
		\(E_{\alpha,\beta}^{\gamma,\kappa}(-\,\cdot)\) is not completely
		monotone for any \(\beta,\gamma>0\).
	\end{proposition}
	
	\begin{proof}
		Stirling's formula applied to \eqref{eq:hypothetical-moments} yields
		\begin{equation}\label{eq:moment-root-general}
			m_n^{1/n}
			\sim e^{\alpha-\kappa}
			\frac{\kappa^\kappa}{\alpha^\alpha}
			n^{\kappa-\alpha}\longrightarrow0.
		\end{equation}
		The support lemma would force \(N\) to be concentrated at zero, and hence
		\(m_1=0\).  Formula \eqref{eq:hypothetical-moments} gives \(m_1>0\), a
		contradiction.
	\end{proof}
	
	The admissible results of \Cref{prop:kernel,prop:boundary} and the three
	exclusion propositions exhaust \((0,\infty)^4\), proving
	\Cref{thm:realization}.  Intersecting with \(\Delta>0\) and applying
	\Cref{prop:phases}(i) proves \Cref{thm:series}.
	
	\section{The Prabhakar specialization and examples}
	
	Setting \(\kappa=1\) in \eqref{eq:series} gives the Prabhakar function
	\[
	E_{\alpha,\beta}^{\gamma}(z)
	=\sum_{n=0}^\infty
	\frac{\Gamma(\gamma+n)}
	{\Gamma(\gamma)n!\Gamma(\beta+\alpha n)}z^n.
	\]
	Here \(\Delta=\alpha>0\), so no realization qualification is needed.
	
	\begin{corollary}[Prabhakar complete-monotonicity characterization]\label{cor:prabhakar}
		For \(\alpha,\beta,\gamma>0\),
		\begin{equation}\label{eq:prabhakar-iff}
			E_{\alpha,\beta}^{\gamma}(-\,\cdot)
			\in\CM(0,\infty)
			\quad\Longleftrightarrow\quad
			0<\alpha\leq1
			\quad\text{and}\quad
			\beta\geq\alpha\gamma.
		\end{equation}
	\end{corollary}
	
	The sufficient region in \eqref{eq:prabhakar-iff} was established in
	\cite[Theorem~1, p.~120]{GorskaEtAl2021}. For \(0<\alpha<1\) and
	\(\beta<\alpha\gamma\), necessity follows directly from
	\Cref{prop:negative-mu} with \(\kappa=1\); the Mellin-zero mechanism is the
	obstruction already identified for Wright densities by Ferreira and Simon
	\cite[discussion following Eq.~(4), p.~310]{FerreiraSimon2024}. At
	\(\alpha=1\), the excluded range \(\beta<\gamma\) is covered by
	\Cref{prop:boundary-excluded}. The supercritical case \(\alpha>1\),
	including \(\beta\geq\alpha\gamma\), is excluded by
	\Cref{prop:supercritical}. When \(\gamma=1\), the corollary reduces to
	Schneider's classical two-parameter if-and-only-if criterion
	\cite{Schneider1996}. We record the Prabhakar result as a specialization of
	the four-parameter theorem; the neighboring Kilbas--Saigo and Le Roy-type
	characterizations cited in the Introduction concern different deformations.
	The present corollary concerns the unweighted function
	\(E_{\alpha,\beta}^{\gamma}(-x)\), not the relaxation kernel
	\(t^{\beta-1}E_{\alpha,\beta}^{\gamma}(-t^\alpha)\)
	\cite{TomovskiPoganySrivastava2014,MainardiGarrappa2015}.
	
	Three elementary specializations illustrate the excluded boundary, the supercritical step, and an admissible endpoint.
	
	\begin{example}[An excluded beta boundary]\label{ex:linear-exp}
		For \((\alpha,\beta,\gamma,\kappa)=(1,1,2,1)\),
		\begin{equation}\label{eq:linear-exp}
			E_{1,1}^{2,1}(-x)=(1-x)e^{-x},
		\end{equation}
		which becomes negative for \(x>1\).  This is the stratum
		\(\alpha=\kappa,\beta<\gamma\).
	\end{example}
	
	\begin{example}[A supercritical step]\label{ex:cosine}
		For \((\alpha,\beta,\gamma,\kappa)=(2,1,1,1)\),
		\begin{equation}\label{eq:cosine}
			E_{2,1}^{1,1}(-x)
			=\sum_{n=0}^\infty\frac{(-x)^n}{(2n)!}
			=\cos\sqrt{x},
		\end{equation}
		which is not nonnegative on \((0,\infty)\), hence not completely monotone.
	\end{example}
	
	\begin{example}[An admissible powered-beta case]\label{ex:one-minus-exp}
		For \((\alpha,\beta,\gamma,\kappa)=(1,2,1,1)\),
		\begin{equation}\label{eq:one-minus-exp}
			E_{1,2}^{1,1}(-x)=\frac{1-e^{-x}}{x}
			=\int_0^1e^{-xu}\dd u.
		\end{equation}
		This is \eqref{eq:beta-density} with the uniform measure on \((0,1)\).
	\end{example}
	
	\section{Conclusion}

We have obtained necessary and sufficient conditions for complete
monotonicity of the Srivastava--Tomovski series in its entire regime and for
its positive-axis realization. The analytic phase and the positivity boundary
are separate: the series phase is governed by \(\Delta=1+\alpha-\kappa\),
whereas positivity of the interior Wright kernel is governed by
\(a=\alpha/\kappa\) and \(\mu=\beta-a\gamma\). In the entire regime the
original series is completely monotone exactly when
\(\alpha\leq\kappa\) and \(\kappa\beta\geq\alpha\gamma\); the same criterion
holds for the explicitly defined Laplace--Wright realization across the
entire, finite-radius, and zero-radius phases.

For \(0<\alpha<\kappa\), the admissible Bernstein measure is the powered
pushforward of a power-biased Wright law. At \(\alpha=\kappa\), it becomes a
compact powered-beta law and, at \(\beta=\gamma\), the atom
\(\Gamma(\gamma)^{-1}\delta_1\). The excluded region \(\mu<0\) is ruled out
by the Mellin-zero obstruction and uniqueness of finite signed Laplace
transforms; the equal-step and supercritical regions are excluded by
moment-support asymptotics. Setting \(\kappa=1\) gives the Prabhakar
specialization.

Because the Srivastava--Tomovski extension was formulated as a
fractional-calculus kernel, the distinction between the defining series and a
separately defined positive-axis realization is not merely terminological.
Outside the entire regime, any use of the Laplace realization should identify
that realization explicitly. Multiplication by powers or nonlinear changes
of argument can also change complete-monotonicity and integrability
conditions and therefore require separate analysis.

	\appendix
	\crefalias{section}{appendix}
	
	\section{Mellin transform of the second-kind Wright function}\label{app:mellin}
	
	The proof of \Cref{lem:wright-mellin} starts from the classical Mellin--Barnes gamma quotient. Luchko's formula
	\cite[Eq.~(4), p.~245]{Luchko2019} contains
	\(\Gamma(s)/\Gamma(\beta-\rho s)\). The contour representation does not
	alone determine the sharp strip of the ordinary Mellin integral. The argument
	below establishes that integral and the absolute-convergence boundary required
	in \Cref{prop:negative-mu}. For the negative-ray asymptotics,
	Wright's theorem treats the exceptional gamma-pole cases
	\cite[Theorem~1, pp.~37--38, and Section~4, pp.~45--46]{Wright1940}; see also
	\cite[Theorem~6 and Eq.~(4.24), p.~393]{Paris2017}. With
	\[
	X=(1-a)(a^at)^{1/(1-a)},
	\]
	the large-\(t\) expansion has an algebraic factor times \(e^{-X}\).  Consequently, for fixed
	\(0<a<1\) and \(\nu\in\R\),
	\begin{equation}\label{eq:appendix-tail}
		\lvert W_{-a,\nu}(-t)\rvert\leq Ct^N
		\exp\{-bt^{1/(1-a)}\},\qquad t\geq1,
	\end{equation}
	for suitable \(C,N,b>0\).  The
	exponential rate is
	\begin{equation}\label{eq:exact-B}
		B=(1-a)a^{a/(1-a)},
	\end{equation}
	in the sense that \eqref{eq:appendix-tail} may be taken with every
	\(b<B\).  The algebraic prefactor depends on \(\nu\), but the rate does not.
	
	At the origin, the constant term is \(1/\Gamma(\nu)\).  If
	\(\nu\notin\mathbb Z_{\leq0}\), this gives the local Mellin boundary
	\(\Rea z>0\).  If \(\nu=-m\in\mathbb Z_{\leq0}\), the constant term
	vanishes and the coefficient of \(t\) is
	\(-1/\Gamma(-m-a)\neq0\), giving \(\Rea z>-1\).
	
	Choose \(N\in\mathbb N_0\) so that \(q=\nu+Na\geq0\).  Termwise
	differentiation of the entire Wright series gives
	\begin{equation}\label{eq:derivative-recurrence}
		\frac{\dd^N}{\dd t^N}W_{-a,q}(-t)
		=(-1)^N W_{-a,\nu}(-t).
	\end{equation}
	For \(q\geq0\), Ferreira and Simon's normalized Wright density has fractional
	moments \cite[Eqs.~(3)--(4), p.~310]{FerreiraSimon2024}
	\[
	\mathbb E[M_{a,q}^{s}]
	=\frac{\Gamma(1+s)\Gamma(a+q)}
	{\Gamma(a+q+as)},
	\qquad s>-1.
	\]
	Dividing by the normalizing factor \(\Gamma(a+q)\) and setting
	\(w=s+1\) yields
	\begin{equation}\label{eq:base-mellin}
		\int_0^\infty t^{w-1}W_{-a,q}(-t)\dd t
		=\frac{\Gamma(w)}{\Gamma(q+aw)},
		\qquad \Rea w>0.
	\end{equation}
	The integral defines a holomorphic function on \(\Rea w>0\): on compact
	subsets, its behavior at the origin together with \eqref{eq:appendix-tail}
	provides an integrable majorant. Since the identity holds for every real
	\(w>0\), the identity theorem extends it to the half-plane.
	
	For \(\Rea z>N\), insert \eqref{eq:derivative-recurrence} and integrate by
	parts \(N\) times.  At the \(j\)-th boundary step, the term is a constant
	multiple of
	\(t^{z-1-j}W_{-a,q-(N-1-j)a}(-t)\), with \(0\leq j<N\).
	It vanishes at infinity by \eqref{eq:appendix-tail}; at zero its power has
	real part at least \(\Rea z-N>0\), while the Wright factor is finite.
	Hence every boundary term vanishes, and
	\begin{align*}
		\int_0^\infty t^{z-1}W_{-a,\nu}(-t)\dd t
		&=\frac{\Gamma(z)}{\Gamma(z-N)}
		\int_0^\infty t^{z-N-1}W_{-a,q}(-t)\dd t\\
		&=\frac{\Gamma(z)}{\Gamma(\nu+az)}.
	\end{align*}
	The integral is holomorphic on \(\Rea z>0\) by its local behavior and
	stretched-exponential tail, so the identity theorem extends the formula to
	that half-plane.
	
	If \(\nu=-m\), the integral is holomorphic on \(\Rea z>-1\).  At zero,
	\[
	\frac1{\Gamma(-m+az)}\sim(-1)^m m!\,az,
	\qquad
	\Gamma(z)\sim z^{-1},
	\]
	so the quotient has the removable value \((-1)^m m!a\).  The identity
	theorem gives the asserted enlarged strip.  This proves
	\Cref{lem:wright-mellin}.
	
	Hence the fundamental half-plane is \(\Rea z>0\) unless
	\(\nu\in\mathbb Z_{\leq0}\), in which case the zero of the Wright function at
	the origin shifts the lower boundary to \(\Rea z>-1\) and cancels the pole at
	\(z=0\). The general Fox--\(H\) Mellin-transform formula and its
	fundamental-strip conditions are given in
	\cite[Theorem~2.2 and Eqs.~(2.5.5)--(2.5.7), pp.~43--44]{KilbasSaigo2004}.
	
	\section{Convergence and realization phases}\label{app:phases}
	
	Let
	\[
	c_n:=\frac{\Gamma(\gamma+\kappa n)}
	{\Gamma(\gamma)n!\Gamma(\beta+\alpha n)}.
	\]
	Gamma-ratio asymptotics give
	\begin{equation}\label{eq:coefficient-ratio}
		\frac{c_{n+1}}{c_n}
		=\frac{\kappa^\kappa}{\alpha^\alpha}n^{-\Delta}
		\left[1+
		\frac{\gamma-\beta-(1+\Delta)/2}{n}
		+O(n^{-2})\right].
	\end{equation}
	This proves \Cref{prop:convergence}.  In the critical case
	\(\Delta=0\), hence \(\kappa=\alpha+1\), and
	\begin{equation}\label{eq:critical-ratio}
		\frac{c_{n+1}}{c_n}
		=R^{-1}\left[1+
		\frac{\gamma-\beta-1/2}{n}+O(n^{-2})\right].
	\end{equation}
	Stirling's formula also gives the
	sharper boundary estimate
	\begin{equation}\label{eq:critical-coefficient}
		c_nR^n
		\sim
		\frac{\kappa^{\gamma-1/2}}
		{\sqrt{2\pi}\,\Gamma(\gamma)\alpha^{\beta-1/2}}
		n^{\gamma-\beta-1/2}.
	\end{equation}
	Consequently, the power series converges absolutely on \(\lvert z\rvert=R\) if
	\(\beta-\gamma>1/2\).  It converges conditionally at boundary points
	\(z\neq R\) when
	\(-1/2<\beta-\gamma\leq1/2\), and its terms fail to tend to zero when
	\(\beta-\gamma\leq-1/2\).  At the positive point \(z=R\), there is no
	conditional-convergence regime.
	Indeed, in the conditional range the exponent
	\(p=\gamma-\beta-1/2\) is negative, and
	\eqref{eq:critical-ratio} shows that \(c_nR^n\) is eventually decreasing to
	zero.  Dirichlet's test therefore applies at every boundary point
	\(z/R\neq1\); no monotonicity is being inferred from the asymptotic
	equivalence alone.
	
	For the Laplace realization, \eqref{eq:appendix-tail} and the substitution
	\(u=t^\kappa\) give
	\begin{equation}\label{eq:kernel-tail-app}
		\lvert K_{\alpha,\beta}^{\gamma,\kappa}(u)\rvert
		\leq C'u^{N'}
		\exp\{-b'u^{1/(\kappa-\alpha)}\},
		\qquad u\geq1.
	\end{equation}
	The exact rate before an arbitrarily small reduction is \(B\) from
	\eqref{eq:exact-B}.  If \(\Delta>0\), then
	\(1/(\kappa-\alpha)>1\), so
	\(\int e^{r u}\lvert K(u)\rvert\dd u<\infty\) for every \(r>0\).  In particular,
	\[
	\int_0^\infty e^{\lvert x\rvert u}\lvert K(u)\rvert\dd u<\infty,
	\qquad x\in\C.
	\]
	Expanding \(e^{-xu}\) under this absolute dominating function and using
	\eqref{eq:kernel-moments} proves entire equality with \eqref{eq:series}.
	
	If \(\Delta=0\), then \(\kappa-\alpha=1\), and
	\begin{equation}\label{eq:B-equals-R}
		B=(1-a)a^{a/(1-a)}
		=\frac{\alpha^\alpha}{\kappa^\kappa}=R.
	\end{equation}
	Thus the integral converges locally uniformly for \(\Rea x>-R\).  For
	\(\lvert x\rvert<R\), absolute domination by
	\(e^{\lvert x\rvert u}\lvert K(u)\rvert\) permits termwise
	expansion and proves agreement with the nonzero series germ.
	The leading coefficient in the Wright asymptotic cited in
	\Cref{app:mellin} is nonzero, so \(-R\) is the abscissa of absolute
	convergence; convergence at individual points of the boundary line can still
	depend on the algebraic prefactor.
	
	If \(\Delta<0\), then \(1/(\kappa-\alpha)<1\).  The integral remains
	holomorphic for \(\Rea x>0\), and \eqref{eq:kernel-tail-app} gives every
	absolute moment, so all right derivatives at zero exist.  Were the function
	analytic at zero, those derivatives would generate its Taylor series with a
	positive radius.  Formula \eqref{eq:coefficient-ratio} shows that this Taylor
	series has radius zero, a contradiction.  This proves \Cref{prop:phases}.
	
	\section{Signed Laplace uniqueness and moment support}\label{app:measure}
	
	\begin{lemma}[Finite signed Laplace uniqueness]\label{lem:signed-unique}
		Let \(\sigma\) be a finite signed Borel measure on \([0,\infty)\).  If
		\[
		\int_{[0,\infty)}e^{-xu}\sigma(\dd u)=0
		\qquad\text{for every }x>0,
		\]
		then \(\sigma=0\).
	\end{lemma}
	
	\begin{proof}
		Write the Jordan decomposition \(\sigma=\sigma^+-\sigma^-\).  The two finite
		positive measures have identical Laplace transforms.  Uniqueness for positive
		measures, for example \cite[Proposition~1.2, p.~2]{SchillingSongVondracek2012},
		gives \(\sigma^+=\sigma^-\).
	\end{proof}
	
	\begin{lemma}[Moment roots determine the upper support endpoint]
		\label{lem:moment-root}
		Let \(N\) be a finite nonzero positive measure on \([0,\infty)\) with all
		moments \(m_n=\int u^nN(\dd u)\) finite.  Then
		\begin{equation}\label{eq:moment-support}
			\lim_{n\to\infty}m_n^{1/n}=\esssup_N u
		\end{equation}
		in \([0,\infty]\), whenever the left side is interpreted in the extended
		sense.  In particular, a finite limit \(L\) forces support in \([0,L]\), and
		\(L=0\) forces support at zero.
	\end{lemma}
	
	\begin{proof}
		If \(N((r,\infty))>0\), then
		\(m_n\geq r^nN((r,\infty))\), so
		\(\liminf m_n^{1/n}\geq r\).  Taking \(r\) upward to the essential support
		endpoint gives the lower bound.  If \(N\) is supported in \([0,L]\), then
		\(m_n\leq L^nN([0,\infty))\), giving the reverse bound.  The infinite-endpoint
		case follows by allowing arbitrary finite \(r\).
	\end{proof}
	
	Applying Stirling's formula directly to
	\eqref{eq:hypothetical-moments} gives
	\[
	m_n^{1/n}\sim
	e^{\alpha-\kappa}\frac{\kappa^\kappa}{\alpha^\alpha}
	n^{\kappa-\alpha}.
	\]
	This is \(1\) when \(\alpha=\kappa\) and tends to zero when
	\(\alpha>\kappa\), exactly as used in \Cref{prop:boundary-excluded,prop:supercritical}.
	
	\section{Endpoint calculations}\label{app:endpoints}
	
	For \(\alpha=\kappa\), \(\beta>\gamma\), the substitution
	\(u=t^\kappa\) transforms \eqref{eq:beta-density} into
	\[
	K_=(u)\dd u
	=\frac{t^{\gamma-1}(1-t)^{\beta-\gamma-1}}
	{\Gamma(\gamma)\Gamma(\beta-\gamma)}\dd t,
	\qquad 0<t<1.
	\]
	Its mass is \(1/\Gamma(\beta)\), and multiplication by
	\(\Gamma(\beta)\) gives the beta probability density.  At
	\(\beta=\gamma\), direct coefficient cancellation gives
	\[
	E_{\kappa,\gamma}^{\gamma,\kappa}(z)
	=\frac1{\Gamma(\gamma)}\sum_{n=0}^\infty\frac{z^n}{n!}
	=\frac{e^z}{\Gamma(\gamma)}.
	\]
	
	The three identities in \Cref{ex:linear-exp,ex:cosine,ex:one-minus-exp} follow
	from
	\begin{align*}
		\frac{\Gamma(n+2)}{n!\Gamma(n+1)}&=n+1,\\
		\frac{\Gamma(n+1)}{n!\Gamma(2n+1)}&=\frac1{(2n)!},\\
		\frac{\Gamma(n+1)}{n!\Gamma(n+2)}&=\frac1{(n+1)!},
	\end{align*}
	respectively.  In particular, the middle identity produces
	\(\cos\sqrt{x}\), because the series variable is \(-x\), not \(-x^2\).

\section*{Statements and Declarations}

\paragraph{Funding} No funding was received for conducting this study.

\paragraph{Competing interests} The authors have no competing interests to declare that are relevant to the content of this article.

\paragraph{Data availability} No datasets were generated or analyzed during the current study because the article is purely theoretical.

\end{document}